\UseRawInputEncoding
\documentclass[a4paper, 12pt,reqno]{amsart}
\usepackage{amsmath,amssymb,amsthm,enumerate}
\allowdisplaybreaks
\theoremstyle{plain}
\newtheorem{theorem}{Theorem}
\newtheorem*{theorem*}{Theorem}
\newtheorem{lemma}{Lemma}
\newtheorem*{lemma*}{Lemma}
\theoremstyle{definition}

\newtheorem*{definition*}{Definition}
\theoremstyle{remark}

\newtheorem{example}{Example}
\newtheorem{statement}{Statement}
\newtheorem*{remark*}{Remark}

\newtheorem*{statement*}{Statement}
\newtheorem{corollary}{Corollary}
\begin{document}
\title[Applications of probability distributions on positive integers]{Applications of  certain probability distributions on positive integers}

\author{Symon Serbenyuk}



\maketitle
\text{\emph{simon6@ukr.net}}\\
\text{\emph{Kharkiv National University of Internal Affairs,}}\\
\text{\emph{L.~Landau avenue, 27, Kharkiv, 61080, Ukraine}}
\begin{abstract}

The main goal of this research is to model and investigate generalizations of functions from  \cite{Symon2024}. Arguments of modeled functions are presented by the representation $\pi_{\mathfrak p}$ from \cite{JN2022}. 

\end{abstract}

\section{Introduction}

Mathematical objects having  a complicated local structure are also called pathological (or ``mathematical monsters") can be characterized by a number of applications of them.  In real analysis,  fractal sets (\cite{Bunde1994, Falconer1997, Falconer2004, Mandelbrot1977, Mandelbrot1999, Moran1946, Symon21, Symon21-1, Symon2021} and references therein), as well as  singular or  some their generalizations (for example, \cite{{Salem1943}, {Zamfirescu1981}, {Minkowski}, {S.Serbenyuk 2017}, 2, 3, 4, 11}), nowhere monotonic, and non-differentiable functions (for example, see \cite{{Bush1952}}, etc.)., are such objects. The interest in the mentioned classes of functions can be explained by their connection  with fractals and fractal multiformalism (for these studies, the motivation  is given in \cite{ALSW2024, ALSW2024a, AS2021, CLS2024, DS2020, D2021, DS2023, DSM2021, Selmi2021, 8,16,  Hensley, Hirst, 2019, 10, 16}, etc.), as well as by the fact that such functions are an important tool for modeling real objects, processes, and phenomena (in physics, chemistry, and biology, as well economics, technology, etc.) and for research  different areas of mathematics (for example, see~\cite{BK2000, ACFS2011, Kruppel2009, OSS1995,    Symon21, Symon21-1, Sumi2009, Takayasu1984, TAS1993, Symon2021}).

Modelling  functions with the complicated local structure (including modeling the simplest examples of such functions) was initiated by the classics of mathematics and was continued in the research of their followers. One can note that  various examples of such classes of functions appeared in the scientific results of Bolzano,  Cantor, Darboux, Dini, Minkowski, Riemann, Salem, Weierstrass  and other scientists.

\section{Definition of the object}

Suppose $\mathbb N$ is the set of all positive integers and $(n_j) \in \mathbb N^{\mathbb N}$, $j=1, 2, 3, \dots$. 

 In~\cite{JN2022}, the following expansion of real numbers from $[0,1)$ was presented:
\begin{equation}
\label{eq}
\widehat{p_{n_1}}+\sum^{\infty} _{j=1}{\left(\widehat{p_{n_{j+1}}}p_{n_1}p_{n_2} \cdots p_{n_j}\right)}:=\pi_{\mathfrak p}((n_j))=\pi_{\mathfrak p} (n_1, n_2, \dots, n_j, \dots),
\end{equation} 
where $\mathfrak p :=(p_j)_{j\in\mathbb N}$ is a probability distribution supported by $\mathbb N$, and for all positive integers $k$, the  conditions  $\widehat{p_{1}}=0$, $n_j\in\mathbb N$, and $p_{n_j}\in (0,1)$, as well as 
$$
\sum^{\infty} _{n_j=1}{p_{n_j}}=1 ~~~~~\text{and}~~~~\widehat{p_{n_j}}=\sum^{n_j-1} _{j=1}{p_{j}}~~~\text{(for}~n_j>1)
$$
hold.

One can note that the mentioned expansion is a certain generalization of expansions of values of the Salem function (the Salem function was introduced in \cite{Salem1943}). 

Let us note the following auxiliary properties of this expansion (\cite{JN2022}).

\begin{statement}{\cite[p. 2]{JN2022}}
For all probability distributions $\mathfrak p$ supported by $\mathbb N$, the map $\pi_{\mathfrak p}: \mathbb N^{\mathbb N}\to [0,1)$ is:
\begin{itemize}
\item  a bijection;
\item continuous.
\end{itemize}
\end{statement}

The main goal of this research is to model and investigate generalizations of functions from  \cite{Symon2024}. Arguments of modeled functions are presented in terms of $\pi_{\mathfrak p}$. 

Suppose we have fixed probability vectors
$$
\mathfrak p := (p_1, p_2, p_3, \dots , p_j ,\dots ) ~~~\text{and}~~~ \mathfrak o :=(o_1, o_2, o_3, \dots , o_j ,\dots ) 
$$
with $p_j, o_j \in (0, 1)$ for all positive integers $j$. Suppose $\phi : \mathbb N \to \mathbb N$ is a fixed map with $\phi (i) \ne \phi (j)$ for any $i\ne j$ and $\{\phi (n): n \in \mathbb N\} =\mathbb N$. Also, for simplifications of notations, let us denote: 
$$
\phi (n_j)=m_j.
$$
Then let us consider a function of the form:
$$
G: \pi_{\mathfrak p}((n_j)) \to \pi_{\mathfrak o}((m_j)).
$$
That is, 
$$
x=\pi_{\mathfrak p}((n_j))=\widehat{p_{n_1}}+ \sum^{\infty} _{j=1}{\left(\widehat{p_{n_{j+1}}}p_{n_1}p_{n_2} \cdots p_{n_j}\right)}  \stackrel{G}{\rightarrow} \widehat{o_{m_1}}+ \sum^{\infty} _{j=1}{\left(\widehat{o_{m_{j+1}}}o_{m_1}o_{m_2} \cdots o_{m_j}\right)}=\pi_{\mathfrak o}((m_j))=G(x)=y.
$$

\begin{example} 
It is easy to see that 
$$
\mathbb N =(1, 2, 3, \dots , j, \dots )~~~\text{and}~~~\phi(\mathbb N)=(\phi(1),\phi(2), \phi(3), \dots , \phi(j), \dots ).
$$

The set of all functions $G$ contains the function $y=x$, where $\phi(n_j)=n_j$. In addition, one can note the second example of $\phi$: 
$$
\phi(n_j)=\begin{cases}
n_j+1 &\text{whenever $n_j$ is odd}\\
n_j-1 &\text{whenever $n_j$ is even.}
\end{cases}
$$
\end{example}

 For  simplifications of notations, one can use an auxiliary notion of the shift operator $\sigma$. That is, for $x=\pi_P((i_k))$, let us define the following
$$
\sigma(x)= \pi_{\mathfrak p}((n_j)\setminus\{n_1\})= \pi_P(n_2, n_3, \dots ).
$$
In the other words, we have
$$
\sigma(x)=\widehat{p_{n_2}}+\sum^{\infty} _{j=2}{\left(\widehat{p_{n_{j+1}}}p_{n_2}p_{n_3} \cdots p_{n_j}\right)}
$$
and
$$
x=\widehat{p_{n_1}}+p_{n_1}\sigma(x).
$$
By analogy, one can write
$$
\sigma^k(x)=\widehat{p_{n_{k+1}}}+\sum^{\infty} _{j=k+1}{\left(\widehat{p_{n_{j+1}}}p_{n_{k+1}}p_{i_{k+2}} \cdots p_{n_j}\right)}
$$
and
$$
\sigma^k(x)=\widehat{p_{n_{k+1}}}+p_{n_{k+1}}\sigma^{k+1}(x)
$$
for $n=0, 1, 2, 3, \dots$, where $\sigma^0(x)=x$.

\section{Properties of functions}

\begin{lemma}
Any function  $G$  has the following properties:
\begin{enumerate}
\item $G$ maps the interval $[0,1)$ into $[0,1)$;
\item the function $G$ is  bijective on the domain;
\item $G$ is not a monotonic function on the domain;
\item $G$ is a continuous on the domain. 
\end{enumerate}
\end{lemma}
\begin{proof}
\emph{The first property} follows from the definition of expansions of the argument and values of the considered functions. 

\emph{The second property}. Since $\pi_{\mathfrak p}$ is bijection and $mathfrak p$ is probability distribution on positive integers, which construct a represetation of real numbers, then  $\pi_{\mathfrak p}$ is an increasing or decreasing function.

Suppose $x_1=\pi_{\mathfrak p}((n_j))$ and $x_2=\pi_{\mathfrak p}((r_j))$ and $n_1< r_1$. Then
$$
x_2-x_1=\widehat{p_{r_1}}+ p_{r_1}\left(\widehat{p_{r_2}}+\sum^{\infty} _{j=2}{\left(\widehat{p_{r_{j+1}}}p_{r_1}p_{r_2} \cdots p_{r_j}\right)}\right) -\widehat{p_{n_1}}+ p_{n_1}\left(\widehat{p_{n_2}}+\sum^{\infty} _{j=2}{\left(\widehat{p_{n_{j+1}}}p_{n_1}p_{n_2} \cdots p_{n_j}\right)}\right) 
$$
$$
=\widehat{p_{r_1}}-\widehat{p_{n_1}}+p_{r_1}\left(\widehat{p_{r_2}}+\sum^{\infty} _{j=2}{\left(\widehat{p_{r_{j+1}}}p_{r_1}p_{r_2} \cdots p_{r_j}\right)}\right) -p_{n_1}\left(\widehat{p_{n_2}}+\sum^{\infty} _{j=2}{\left(\widehat{p_{n_{j+1}}}p_{n_1}p_{n_2} \cdots p_{n_j}\right)}\right) 
$$
$$
=p_{n_1}+p_{n_1+1}+\dots + p_{r_1-1}+p_{r_1}\left(\widehat{p_{r_2}}+\sum^{\infty} _{j=2}{\left(\widehat{p_{r_{j+1}}}p_{r_1}p_{r_2} \cdots p_{r_j}\right)}\right) -p_{n_1}\left(\widehat{p_{n_2}}+\sum^{\infty} _{j=2}{\left(\widehat{p_{n_{j+1}}}p_{n_1}p_{n_2} \cdots p_{n_j}\right)}\right) 
$$
$$
=p_{n_1+1}+\dots + p_{r_1-1}+p_{n_1}\left(1-\widehat{p_{n_2}}-\sum^{\infty} _{j=2}{\left(\widehat{p_{n_{j+1}}}p_{n_1}p_{n_2} \cdots p_{n_j}\right)}\right) +p_{r_1}\left(\widehat{p_{r_2}}+\sum^{\infty} _{j=2}{\left(\widehat{p_{r_{j+1}}}p_{r_1}p_{r_2} \cdots p_{r_j}\right)}\right)>0.
$$

By induction, for any $k\in \mathbb N$ for ${n_u}={r_u}$, where $u=\overline{1, k}$, and $n_{k+1}< r_{k+1}$, we obtain $x_2-x_1>0$.

So, $\pi_{\mathfrak p}$ is an increasing map. 

Since $\pi_{\mathfrak p}$,  $\pi_{\mathfrak o}$, and $\phi$ are  bijections, as well as any $x\in [0, 1)$ has a unique expansion by $\pi_{\mathfrak p}$, then the second property is true.

Let us prove \emph{the third and  fourth properties}. Suppose $x_1<x_2$.  Then there exists $u\in \mathbb N$ such that ${n_i}={r_i}$ holds for  $i=\overline{1, u}$ and  $n_{u+1}< r_{u+1}$. Hence  $\phi(n_i)=\phi(r_i)$ holds for  $i=\overline{1, u}$ but $\phi (n_{u+1})<\phi (r_{u+1})$ or $\phi (n_{u+1})>\phi (r_{u+1})$ as well. 

So, $f$ is not monotonic. 

Now let us consider the difference of the mentioned values of the function. That is, 
$$
G(x_2)-G(x_1)=\pi_{\mathfrak 0}((\phi(r_j)))-\pi_{\mathfrak 0}((\phi(n_j)))=o_{m_1}o_{m_2}\dots o_{m_u}\left(\sigma^u \left(\pi_{\mathfrak 0}((\phi(r_j)))\right)-\sigma^u \left(\pi_{\mathfrak 0}((\phi(n_j)))\right)\right)
$$
$$
< (1-0)p_{m_1}p_{m_2}\dots p_{m_u} \to 0 ~~~(u \to \infty).
$$
Here $\sigma$ is the  shift operator, $0<o_j<1$, and $\sigma^k(\pi_{\mathfrak 0}((n_j))), \sigma^k(\pi_{\mathfrak p}((n_j)))\in [0, 1)$.

Whence
$$
\lim_{u\to \infty}{|G(x_2)-G(x_1)|}=\lim_{u\to \infty}{o_{m_1}o_{m_2}\dots o_{m_u}}=0
$$
and
$$
\lim_{x_2\to x_1}{G(x_2)}=G(x_1).
$$
So, $G$ is a continuous function at any $x\in [0,1)$.
\end{proof}

\begin{theorem}
At the point $x_0=\pi_{\mathfrak p}((n_j))$, the function $G$:
\begin{itemize}
\item is a singular function whenever there exists only a finite subsequence $(n_k)$ of positive integers such that  the inequality ${o_{\phi(n_k)}}\ge {p_{n_k}}$ holds.
\item has an  infinite derivative whenever there exists only a finite subsequence $(n_k)$ of positive integers such that  the inequality ${o_{\phi(n_k)}}\le {p_{n_k}}$ holds.
\item has a finite derivative whenever there exists only a finite subsequence $(n_k)$ of positive integers such that  the inequality ${o_{\phi(n_k)}}\ne {p_{n_k}}$ holds.
\end{itemize}
\end{theorem}
\begin{proof}
Let us consider intervals of the forms (auxiliary analogy is also used for  sets  in \cite{JN2022}):
$$
[x_k, x_{k+1}):=\left[\pi_{\mathfrak p}((n_1, n_2, \dots , n_{k-1}, n_k, 1, 1, 1, \dots )), \pi_{\mathfrak p}((n_1, n_2, \dots , n_{k-1}, n_k+1, 1, 1, 1, \dots ))\right)
$$
We have
$$
x_{k+1}-x_k=p_{n_1}p_{n_2}\dots p_{n_{k-1}}\left(\widehat{p_{n_{k}+1}}-\widehat{p_{n_k}}+(p_{n_{k}+1}-p_{n_k})G((1))\right)=p_{n_1}p_{n_2}\dots p_{n_{k-1}}p_{n_{k}}.
$$
Then
$$
G(x_{k+1})-G(x_k)=o_{m_1}o_{m_2}\dots o_{m_{k-1}}\left(\widehat{o_{m_{k}+1}}-\widehat{o_{m_k}}+(o_{m_{k}+1}-o_{m_k})G((1))\right).
$$

Whence,
$$
\lim_{k \to \infty}{\frac{G(x_{k+1})-G(x_k)}{x_{k+1}-x_k}}=\lim_{k \to \infty}{{o_{m_1}o_{m_2}\dots o_{m_{k-1}}\left(\widehat{o_{m_{k}+1}}-\widehat{o_{m_k}}+(o_{m_{k}+1}-o_{m_k})G((1))\right)} }{ p_{n_1}p_{n_2}\dots p_{n_{k-1}}p_{n_{k}}}.
$$
So,  the derivative depends on the limit
$$
\lim_{k \to \infty}\frac{{o_{m_1}o_{m_2}\dots o_{m_{k}}}}{ p_{n_1}p_{n_2}\dots p_{n_{k-1}}p_{n_{k}}}.
$$

Suppose   $\sharp_j(x, k)$ is the number of the digit $j$ in the $k$ first digits of the $\pi_{\mathfrak p}$ representation of $x$. Then  
$$
\lim_{k \to \infty}\frac{{o_{m_1}o_{m_2}\dots o_{m_{k}}}}{ p_{n_1}p_{n_2}\dots p_{n_{k-1}}p_{n_{k}}}=\lim_{k \to \infty}\frac{{o_{\phi(n_1)}o_{\phi(n_2)}\dots o_{\phi(n_k)}}}{ p_{n_1}p_{n_2}\dots p_{n_{k-1}}p_{n_{k}}}=\lim_{k\to\infty} \left(\frac{o^{\sharp_1(x, k)} _{\phi(1)} o^{\sharp_2(x, k)} _{\phi(2)} \cdots  o^{\sharp_{j}(x,k)} _{\phi (j)} \cdots }{p^{\sharp_1(x, k)} _1 p^{\sharp_2(x, k)} _2 \cdots p^{\sharp_{j}(x, k)} _{j}\cdots}  \right).
$$

So,  coordinates of the probability vectors $\mathfrak p$ and $\mathfrak o$, as well as the rule $\phi$ are key parameters of the existence of the derivative. Moreover, the main parameter is following:
$$
\frac{o_{\phi(j)}}{p_j}.
$$
\end{proof}

\begin{corollary} It follows that in the class of the presented functions, there exist singular functions  $G$  on $[0, 1)$ such that the functions
$$
G^{-1}: \pi_{\mathfrak o}((\phi(n_j))) \to \pi_{\mathfrak p}((n_j))
$$
have an infinite derivative  almost everywhere on~$[0, 1)$ and vice versa (there exist such functions $G$ with an infinite derivative almost everywhere that $G^{-1}$ are singular functions). 
\end{corollary}

\begin{example}
Suppose 
$$
{\mathfrak p}=\left(\frac{1}{2}, \frac{1}{4}, \frac{1}{8}, \dots, \frac{1}{2^j}, \dots \right),
$$
$$
{\mathfrak o}=\left(\frac{1}{3}, \frac{2}{9}, \frac{4}{27}, \dots, \frac{2^{j-1}}{3^j}, \dots \right),
$$
and
$$
\phi(n_j)=\begin{cases}
n_j+1 &\text{whenever $n_j$ is odd}\\
n_j-1 &\text{whenever $n_j$ is even.}
\end{cases}
$$
That is, 
$$
G: \begin{cases}
p_{n_{2t-1}}=\frac{1}{2^{2t-1}} \to o_{\phi(n_{2t-1})}=\frac{2^{2t-1}}{3^{2t}} \\
p_{n_{2t}}=\frac{1}{2^{2t}} \to o_{\phi(n_{2t})}=\frac{2^{2t-2}}{3^{2t-1}},
\end{cases}
$$
where $t=1, 2, 3, \dots $.

Since 
$$
\frac{o_{\phi(n_{2t-1})}}{p_{n_{2t-1}}}=\frac{\frac{2^{2t-1}}{3^{2t}}}{\frac{1}{2^{2t-1}}}=\frac{2^{4t-2}}{3^{2t}}>1~~~\text{and}~~~\frac{o_{\phi(n_{2t})}}{p_{n_{2t}}}=\frac{\frac{2^{2t-2}}{3^{2t-1}}}{\frac{1}{2^{2t}} }=\frac{2^{4t-2}}{3^{2t-1}}>1.
$$

So, $G$ has an infinite derivative almost everywhere on $[0, 1)$, but $G^{-1}$ is singular on $[0, 1)$.
\end{example}

\begin{lemma}  A   system of functional equations of the form
\begin{equation*}
f\left(\sigma^{k-1}(x)\right)=\widehat{o_{m_k}}+o_{m_k}f\left(\sigma^k(x)\right),
\end{equation*}
has the unique solution
$$
G(x)=\widehat{o_{n_1}}+ \sum^{\infty} _{j=1}{\left(\widehat{o_{n_{j+1}}}o_{n_1}o_{n_2} \cdots o_{n_j}\right)}
$$
in the class of determined and bounded on $[0, 1)$ functions. Here $x$ represented by expansion \eqref{eq} induced by the probability distribution on $\mathbb N$, $i_n\in \mathbb N$, and $k=1,2, \dots$, as well as $\sigma$ is the shift operator with $\sigma^0(x)=x$.
\end{lemma}

The statement can be proven by analogy with \cite{Symon2024}.

\begin{theorem}
For the Lebesgue integral, the following equality holds:
$$
\int^1 _0 {G(x)dx}=\frac{\sum^{\infty}_{j=1}{\widehat{o_{\phi(j)}}p_j}}{1-\sum^{\infty} _{j=1}{o_{\phi(j)}p_j}}.
$$
\end{theorem}
\begin{proof}
 Let us begin with some useful equalities:
$$
x=\widehat{p_{n_1}}+p_{n_1}\sigma(x)
$$
and
$$
dx=p_{n_1}d(\sigma(x)),
$$
as well as
$$
d(\sigma^{k-1} (x))=p_{n_k}d(\sigma^{k} (x))
$$
for all $k=1, 2, 3, \dots$ . 

So,
$$
I:=\int^1 _0{G(x)dx}=\sum^{\infty} _{j=1}{\int^{\widehat{p_{j+1}}} _{\widehat{p_{j}}}{G(x)dx}}=\sum^{\infty} _{j=1}{\int^{\widehat{p_{j+1}}} _{\widehat{p_{j}}}{(\widehat{o_{\phi(j)}}+o_{\phi(j)}G(\sigma(x))}dx}
$$
$$
=\sum^{\infty} _{j=1}{\widehat{o_{\phi(j)}}p_{j}} +
 \sum^{\infty} _{j=1}{\int^{\widehat{p_{j+1}}} _{\widehat{p_{j}}}{o_{\phi(j)}G(\sigma(x))dx}}
$$
$$
=\sum^{\infty} _{j=1}{\widehat{o_{\phi(j)}}p_{j}} +
\sum^{\infty} _{j=1}{o_{\phi(j)}\int^{\widehat{p_{j+1}}} _{\widehat{p_{j}}}{p_jG(\sigma(x))d(\sigma(x))}}
$$
$$
=\sum^{\infty} _{j=1}{\widehat{o_{\phi(j)}}p_{j}} +
\sum^{\infty} _{j=1}{o_{\phi(j)}p_j\int^{\widehat{p_{j+1}}} _{\widehat{p_{j}}}{G(\sigma(x))d(\sigma(x))}}.
$$

Using self-affine properties of $f$, we get
$$
I=\sum^{\infty} _{j=1}{\widehat{o_{\phi(j)}}p_{j}}+I\sum^{\infty} _{j=1}{o_{\phi(j)}p_j}.
$$

Finally,
$$
I=\frac{\sum^{\infty} _{j=1}{\widehat{o_{\phi(j)}}p_{j}}}{1-\sum^{\infty} _{j=1}{o_{\phi(j)}p_j}}.
$$
\end{proof}

\section*{Statements and Declarations}
\begin{center}
{\bf{Competing Interests}}

\emph{The author states that there is no conflict of interest.}
\end{center}

\begin{center}
{\bf{Data Availability Statement}}

\emph{There are not suitable for this research.}
\end{center}

\end{document}